\documentclass{amsart}

\usepackage{amssymb}
\usepackage{times}
\usepackage{amsmath}
\usepackage{amsthm}
\usepackage{graphicx}
\usepackage{cases}
\usepackage{cite}

\newtheorem{theorem}{Theorem}[section]
\newtheorem{lemma}[theorem]{Lemma}
\newtheorem{proposition}[theorem]{Proposition}
\newtheorem{corollary}[theorem]{Corollary}
\newtheorem{remark}[theorem]{Remark}
\renewcommand{\P}{\mathord{\mathbb  P}}

\begin{document}

\title{Degeneration of Fermat hypersurfaces in positive characteristic }
\author{THANH HOAI HOANG}
\address{Department of Mathematics, Graduate School of Science, Hiroshima University, 1-3-1 Kagamiyama, Higashi-Hiroshima, 739-8526 JAPAN.}
\email{hoangthanh2127@yahoo.com}
\begin{abstract}
We work over an algebraically closed field $k$ of positive characteristic $p$. Let $q$ be a power of $p$. 
Let $A$ be an $(n+1)\times (n+1)$ matrix with coefficients $a_{ij}$ in $k$, and let $X_A$ be a hypersurface of degree $q+1$ in the projective space $\P ^n$ defined by $\sum a_{ij}x_ix_j^q=0$. It is well-known that if the rank of $A$ is $n+1$, the hypersurface $X_A$ is projectively isomorphic to the Fermat hypersuface of degree $q+1$. We investigate the hypersurfaces $X_A$ when the rank of $A$ is $n$, and determine their projective isomorphism classes. 
\end{abstract}
\keywords{Degeneration, Fermat hypersurface, positive characteristic}
\subjclass{Primary 14J70, secondary 14J50}
\maketitle
\section {Introduction}
We work over an algebraically closed field $k$ of positive characteristic $p$. Let $q$ be a power of $p$. Let $n$ be a positive integer. We denote by $M_{n+1}(k)$ the set of square matrices of size $n+1$ with coefficients in $k$. For a nonzero matrix  $A=(a_{ij})_{0\leq i,j\leq n} \in M_{n+1}(k)$, we denote by $X_A$ the hypersurface of degree $q+1$ defined by the equation $$\sum a_{ij}x_i x_j^q =0$$ in the projective space $\P ^n$ with homogeneous coordinates $(x_0,x_1,\dots ,x_n)$. The following is well-known (\cite{Beau86}, \cite{Lang56}, \cite{Shi01}, see also \S 4 of this paper).
\begin{proposition}\label{shimadaprop}
Let $A=(a_{ij})_{0 \leq i,j \leq n}\in M_{n+1}(k)$ and $X_A\subset \P ^n$ be as above. Then the following conditions are equivalent:
\begin{enumerate}
\item[(i)]  ${\rm rank}(A) =n+1,$
\item[(ii)] $X_A$ is smooth,
\item[(iii)] $X_A$ is isomorphic to the Fermat hypersurface of degree $q+1$, and
\item[(iv)] there exists a linear transformation of coordinates $T \in GL_{n+1}(k)$ such that $^tTAT^{(q)}=I_{n+1}$, where ${^t}T$ is the transpose of $T$, $T^{(q)}$ is the matrix obtained from $T$ by raising each coefficient to its $q$-th power, and $I_{n+1}$ is the identity matrix. 
\end{enumerate}
\end{proposition}
The Fermat hypersurface of degree $q+1$ defined over an algebraically closed field of positive characteristic $p$ has been a subject of numerous papers. It has many interesting properties, such as supersingularity (\cite{Shioda74},\cite{ShioKat},\cite{Tate65}), or unirationality (\cite{Shi92},\cite{Shioda74},\cite{ShioKat}). Moreover, the hypersurface $X_A$ associated with the matrix $A$ with coefficients $a_{ij}$ in the finite field $\mathbb F_{q^2}$, which is called a Hermitian variety, has also been studied for many applications, such as coding theory (\cite{TomLint}). (The general results on Hermitian varieties are due to Segre \cite{Segre65}; see also \cite{Hirsch98}). Therefore it is important  to extend these studies to degenerate cases. 

In the case where characteristic $p\neq 2$, the following is well-known and can be found in any standard textbook on quadratic forms: the hypersurface defined by the quadratic form $\sum a_{ij}x_ix_j=0$ is projectively isomorphic to the hypersurface defined by
$$x_0^2+\cdots +x_{r-1}^2=0, $$ where $r$ is the rank of $A=(a_{ij})$. This result has been extended the case of characteristic 2 (see \cite{Dolgachev12}). Therefore we have a question what is the normal form of the  hypersurfaces defined by a form $\sum a_{ij}x_ix_j^q=0$. 
When $A$ satisfies ${^t}A=A^{(q)}$ and hence this form is the Hermitian form over $\mathbb{F} _q$, the hypersurface $X_A$ is projectively isomorphic over $\mathbb{F}_{q^2}$ to $$x_0^{q+1}+\cdots +x_{r-1}^{q+1}=0,$$
where $r$ is the rank of $A$ (\cite{Hirsch91}). 

In this paper, we classify the hypersurfaces $X_A$ associated with the matrices $A$ of rank $n$ \emph{over an algebraically closed field}. Note that two hypersurfaces $X_A, \ X_{A'}$ associated with the matrices $A, A'$ are projectively isomorphic if and only if there exists a linear transformation $T\in GL_{n+1}(k)$ such that $A' =$  $^{t}TAT^{(q)}$. In this case, we denote $A\sim A'$.

We define $I_s$ to be the $s\times s$ identity matrix, and $E_r$ to be the $r\times r$ matrix 
$$\begin{pmatrix}
0 & 0 & \cdots & 0 \\
1 & 0 & \cdots & 0 \\
\vdots & \ddots & \ddots & \vdots \\
0 & \cdots & 1 &0 \\
\end{pmatrix}.$$
In particular, $E_1=(0)$ and $E_0$ is the $0\times 0$ matrix. Throughout this paper, a blank in a block decomposition of a matrix means that all the components of the block are $0$. Our main result is as follow.

\begin{theorem}\label{maintheorem}
Let $A=(a_{ij})_{0\leq i,j\leq n}$ be a nonzero matrix in $M_{n+1}(k)$, and let $X_A$ be the hypersurface of degree $q+1$ defined by $\sum a_{ij}x_i x_j^q =0$ in the projective space $\P ^n$ with homogeneous coordinates $(x_0,x_1,\dots ,x_n)$. Suppose that the rank of A is $n$. Then the hypersurface $X_A$ is projectively isomorphic to one of the hypersurfaces $X_s$ associated with the matrices
\begin{equation}
W_s= \left ( 
\begin{array}{c|c}
I_s &  \\ \hline
 & E_{n-s+1}\\
\end{array} \right ) \nonumber ,
\end{equation}
where $0\leq s \leq n$. Moreover, if $s\neq s'$, then $X_s$ and $X_{s'}$ are not projectively isomorphic.
\end{theorem}

\begin{corollary}\label{cor1}
If $A$ is a general point of $\{A\in M_{n+1}(k) |\;  \text{rank}(A)=n\}$, then $A\sim W_{n-1}$.
\end{corollary}

\begin{corollary}\label{cor2}
Suppose that $n\geq 2, s<n$ and $(n,s)\neq (2,0)$. Then $X_s$ is rational.
\end{corollary}

We also determine the automorphism group
$${\rm Aut}(X_s)=\{ g\in PGL_{n+1}(k) \mid g(X_s)=X_s \},$$
of the hypersurface $X_s$ for each $s$. For $M\in GL_{n+1}(k)$, we denote by $[M]\in PGL_{n+1}(k)$ the image of $M$ by the natural projection.

\begin{theorem}\label{auttheorem}
Let $X_s$ be the hypersurface associated with the matrix $W_s$ in the projective space $\P ^n$. The projective automorphism group ${\rm Aut}(X_s)$ with $s\leq n-2$ is the group consisting of $[M]$, with
$$M=\left ( 
\begin{array}{c|c|c}
T  & {^t}\mathbf a & 0 \\ \hline
0  & d & 0 \\ \hline
\mathbf c  & e & 1 \\
\end{array} \right ) ,$$
where $T\in GL_{n-1}(k)$, $\mathbf a,\mathbf c$ are row vectors of dimension $n-1$, and $d,e\in k$, and they satisfy the following conditions:
\begin{enumerate}
\item[(i)] $[T]\in {\rm Aut}(X_s^{n-2})$, ${^t}TW'_{s}T^{(q)}=\delta W'_{s}$, $\delta =\delta ^q\neq 0$, where $X_s^{n-2}$ is the hypersurface defined in $\P ^{n-2}$ by the matrix
$$W'_s=\left ( \begin{array}{c|c} I_s &  \\ \hline  & E_{n-s-1} \\ \end{array} \right )$$
\item[(ii)] $d=\delta ,$
\item[(iii)] $[\mathbf aW'_s+d(0, \cdots ,0,1)]\cdot T^{(q)}=\delta (0,\cdots ,0,1)$,
\item[(iv)] ${^t}TW'_s \cdot {^t}\mathbf a^{(q)}+{^t}\mathbf cd^q=0$,
\item[(v)] $[\mathbf aW'_s+d(0,\cdots ,0,1)]\cdot {^t}\mathbf a^{(q)}+ed^q=0.$
\end{enumerate}
Moreover, we have
$${\rm Aut}(X_n)=\left \{ \left [ \begin{array}{c|c} T_n &  \\ \hline \mathbf u & 1 \\ \end{array} \right ] \left | \ \  \parbox[c]{6cm}{${^t}T_nT_n^{(q)}=\lambda I_n, T_n\in GL_n(k), \ \lambda \neq 0, \\ \mathbf u \ \text{is a row vector of dimension}\ n $ } \right. \right \} ,$$ and 
$${\rm Aut}(X_{n-1})=\left \{ \left [ \begin{array}{c|c|c} T_{n-1} &  &  \\ \hline  & \beta &  \\ \hline  &  &1 \\ \end{array} \right ]\left |  \ \ \parbox[c]{5cm}{${^t}T_{n-1}T_{n-1}^{(q)}=\beta ^qI_{n-1}, \\ T_{n-1}\in GL_{n-1}(k),\ 0\neq \beta \in k $} \right. \right \} $$
\end{theorem}

We give a brief outline of our paper. In \S 2, we prove Theorem \ref{maintheorem} and its corollaries. In \S 3, we prove Theorem \ref{auttheorem}. In \S 4, we recall the proof of Proposition \ref{shimadaprop} because this proposition plays an important role in the proof of Theorem \ref{maintheorem}. In \S 5, we investigate the plane curve $X_A$ associated with the matrix $A$ of rank $\leq 2$ in the projective plane $\P ^2$, and recover Homma's unpublished work \cite{Homma} (see Remark \ref{remarkHomma}).

\section{Proofs of Theorem \ref{maintheorem} and its corollaries }
We present several preliminary lemmas. The following remark may be helpful in reading the proof of lemmas.
\begin{remark}
Let 
$$T=\left ( \begin{array}{ccc}
t_{00} & \cdots & t_{0n} \\
\vdots &  & \vdots \\
t_{n0} & \cdots & t_{nn}\\
\end{array} \right ) $$
be an invertible matrix. Suppose that $\sum a_{ij}x_ix_j^q=0$ is the equation associated to the matrix $A=(a_{ij})_{0\leq i,j \leq n}$. Then the operation $$A\mapsto ^{t}TAT^{(q)}$$ on the matrix is equivalent to the transformation of the coordinates 
$$x_i\mapsto \sum_{j=0}^n t_{ij}x_j, $$
where $0\leq i\leq n$.
\end{remark}

\begin{lemma}\label{lemma1}
Put 
\begin{equation}
G_{s,r}=\left ( \begin{array}{c|c|c}
I_s            &                     &                       \\ \hline
               & E_r                &                      \\ \hline  
\mathbf a & 0 \cdots 0 \ 1 &                  \\ 
0             & 0  &                    \\ 
\vdots     &  \vdots  &      E_{n-s-r+1}  \\
0             & 0  &             \\ 
\end{array} \right ) ,\nonumber 
\end{equation}
and
\begin{equation}
G_{s,r+2}=\left ( \begin{array}{c|c|c}
I_s            &                     &                       \\ \hline
               & E_{r+2}           &                      \\ \hline  
\mathbf a^{(q^2)} & 0 \cdots 0 \ 1 &                  \\ 
0             & 0  &                    \\ 
\vdots     &  \vdots  &      E_{n-s-r-1}  \\
0             & 0  &             \\ 
\end{array} \right ) ,\nonumber 
\end{equation}
where $s\geq 1, r\geq 0, n-s-r-1\geq 0$, and $\mathbf a$ is a nonzero row vector of dimension $s$. Then 
$$G_{s,r}\sim G_{s,r+2}.$$
\end{lemma}

\begin{proof}
By the transformation
\begin{equation}
T_G=\left ( \begin{array}{c|c|c|c|c}
I_s            &               & - {^t}\mathbf a  &          &           \\ \hline
               & I_r            &          &               &    \\ \hline
               &               & 1        &               &    \\ \hline
\mathbf a^{(q)}      &                 &         &   1          &    \\ \hline
              &                 &         &              &   I_{n-s-r-1} \\ 
\end{array} \right ) \nonumber ,
\end{equation}
we have $${^t}T_GG_{s,r}T_G^{(q)}=G_{s,r+2}.$$
\end{proof}

\begin{remark}\label{remark1}
Lemma \ref{lemma1} holds when $r=0$ or $n-s-r-1=0$. In particular, when $n-s-r-1=0$, we have $G_{s,r+2}=W_s$.
\end{remark}
\begin{lemma}\label{lemma2}
Put
\begin{equation}
H_{s,r}=\left ( \begin{array}{c|c|c|c}
D_{s-1}  & -{^t}\mathbf a'' \ 0 \cdots 0 &      &                      \\ \hline
-\mathbf a'         &                                &         &                    \\ 
0         &                                &         &                     \\
\vdots &   E_r                         &          &                     \\
0        &                                &         &                      \\ \hline
          &  0 \cdots 0 \ 1          &1         &                    \\ \hline
         &                                &  1       &                    \\ 
         &                                & 0         &                  \\ 
         &                                & \vdots  & E_{n-s-r+1}  \\
         &                                &  0        &                  \\ 
\end{array} \right ) \nonumber ,
\end{equation}
where $s\geq 1, r\geq 2, n-s-r-1\geq 1$, $D_{s-1}\in M_{s-1}(k)$, $\mathbf a'$ and $\mathbf a'' $ are row vectors of dimension $s-1$. Then $$H_{s,r}\sim H_{s,r+2}.$$
\end{lemma}

\begin{proof}
By the transformation 
\begin{equation}
T_H=\left ( \begin{array}{c|c|c|c|c}
I_{s+r-1}  &      &      &      &           \\ \hline
             & 1    &      &      &                 \\ \hline
            & -1   & 1   & 1   &                  \\ \hline
            &       &     &  1   &                  \\ \hline
            &       &     &      &  I_{n-s-r-1}   \\ 
\end{array} \right ) \nonumber ,
\end{equation}
we have 
$${^t}T_HH_{s,r}T_H^{(q)}=H_{s,r+2}.$$  
\end{proof}

\begin{lemma}\label{lemma3}
Put
\begin{equation}
H'_{s,r}=\left ( \begin{array}{c|c|c|c|c|c}
D_{s-1} &           &                   &     &             & \\ \hline
-\mathbf a'  & 0         &                  &     &            & \\ \hline 
          & 1         &                  &     &             &  \\ 
          & 0         &                  &     &            & \\
          & \vdots & E_r             &     &            & \\
          & 0        &                   &     &           &  \\ \hline
          &          & 0 \cdots 0 \ 1&0  &1         &     \\ \hline
          &          &                    & 1   & 0        & \\ \hline
          &          &                    &     & 1         &   \\
          &          &                    &     & 0         &    \\ 
          &          &                    &     & \vdots    & E_{n-s-r-1}  \\
          &          &                    &     & 0           &                  \\ 
\end{array} \right ) \nonumber, 
\end{equation}
where $s\geq 1, r\geq 2, n-s-r-3\geq 1$, $D_{s-1}\in M_{s-1}(k)$, and $\mathbf a'$ is a row vector of dimension $s-1$. Then 
$$H'_{s,r}\sim H'_{s,r+2}.$$
\end{lemma}

\begin{proof}
By the transformation
\begin{equation}
T_{H'}=\left ( \begin{array}{c|c|c|c|c|c}
I_{s+r}     &    &     &   &     &       \\ \hline
             & 1 &    &    &      &       \\ \hline
             &    &  1 &    & 1  &     \\ \hline
             & -1&     & 1 &      &     \\ \hline
             &    &     &    & 1   &      \\ \hline
             &    &    &    &      & I_{n-s-r-3}    \\ 
     
\end{array} \right ) \nonumber ,
\end{equation}
we have
$${^t}T_{H'}H'_{s,r}T_{H'}^{(q)}=H'_{s,r+2}.$$ 
\end{proof}

\begin{remark}
Lemma \ref{lemma2} and ~\ref{lemma3} will be used only in the case where $n-s+1$ is odd. Hence we do not need to prove the case $n-s-1=0$ in Lemma \ref{lemma2} nor the case $n-s-3=0$ in Lemma \ref{lemma3}.
\end{remark}

\begin{lemma}\label{lemma4}
Put 
\begin{equation}
P_s= \left ( 
\begin{array}{c|c}
I_s &  \\ \hline
\mathbf a &    \\
0  &    \\
\vdots & E_{n-s+1}\\
0  & \\
\end{array} \right ) \nonumber,
\end{equation}
where $s\geq 1, n-s+1 \geq 1,$ and $\mathbf a $ is a nonzero row vector of dimension $s$. Then
\begin{enumerate}
\item If $n-s+1$ is even, then $P_s \sim W_s.$
\item If $n-s+1$ is odd, then
\begin{equation}
P_s \sim B_{s-1}= \left ( 
\begin{array}{c|c}
D_{s-1}   &         \\ \hline
\mathbf b_{s-1} &    \\  
      0    &    \\
  \vdots & E_{n-s+2} \\
     0      &   \\
\end{array} \right ) \nonumber,
\end{equation}
\end{enumerate}
where $D_{s-1}\in M_{s-1}(k)$, $\mathbf b_{s-1}$ is the row vector of dimension $s-1$. 
In particular, if $s=1$ and $n$ is odd, then $P_1\sim W_0$.
\end{lemma}

\begin{proof}
(1) Suppose that $n-s+1$ is even. Using Lemma \ref{lemma1} and Remark \ref{remark1}, we have $$P_s=G_{s,0}\sim G_{s,n-s+1}=W_s.$$

(2) Next suppose that $n-s+1$ is odd. By interchanging the coordinates $x_0, \cdots , x_{s-1}$, and scalar multiplication of the coordinates $x_s, \cdots , x_n$ if nessesary, we can show that 
\begin{equation}
P_s \sim P'_s= \left ( 
\begin{array}{c|c|c|c}
I_{s-1}    &           &     &  \\ \hline
            & 1         &    &  \\ \hline
 \mathbf a'  & 1         & 0  &  \\ \hline
            &          & 1    &  \\
            &          & 0  &  \\
            &          & \vdots & E_{n-s}\\   
            &         & 0        &  \\
\end{array} \right ) \nonumber,
\end{equation}
with $\mathbf a'$ being a row vector of dimension $s-1$. By the transformation
\begin{equation}
T_1=\left ( \begin{array}{c|c|c|c}
I_{s-1}    &            &     &  \\ \hline 
-\mathbf a''   & 1         &    &  \\ \hline
            &            & 1  &  \\ \hline
            &            &     &  I_{n-s}\\     
\end{array} \right ) \nonumber,
\end{equation}
with $\mathbf a''^{(q)}=\mathbf a'$, we have
\begin{equation}
Q_s={^t}T_1P'_sT_1^{(q)}= \left ( 
\begin{array}{c|c|c|c}
D_{s-1}   &  -{^t}\mathbf a''  &     &  \\ \hline 
-\mathbf a'   & 1                &     &  \\ \hline
            & 1                & 0  &  \\ \hline
            &                  & 1    &  \\
            &                  & 0  &  \\
            &                  & \vdots & E_{n-s}\\   
            &                 & 0        &  \\         
\end{array} \right ) \nonumber,
\end{equation}
where $D_{s-1}=I_{s-1}+{^t}\mathbf a'' \cdot \mathbf a'$. If $n-s+1=1$, by the transformation 
\begin{equation}
T_2=\left ( 
\begin{array}{c|c|c}
I_{n-1}    &    &       \\ \hline 
            & 1 &      \\ \hline
\mathbf a''  &-1 & 1   \\      
\end{array} \right ) \nonumber,
\end{equation}
we have
$${^t}T_2Q_nT_2^{(q)}=B_{n-1}.$$
Suppose that $n-s+1>1$. Note that, since we are in the case where $n-s+1$ is odd, we have $n-s+1\geq 3$. By the transformation
\begin{equation}
T_3=\left ( 
\begin{array}{c|c|c|c|c}
I_{s-1}    &    &     &  & \\ \hline 
            & 1 &     &  & \\ \hline
            &-1& 1  &1  & \\ \hline
            &    &    &1   &  \\ \hline
            &    &    &   & I_{n-s-1} \\         
\end{array} \right ) \nonumber,
\end{equation}
we have
\begin{equation}
Q'_s={^t}T_3Q_sT_3^{(q)}=\left ( 
\begin{array}{c|c|c|c|c}
D_{s-1}   &-{^t}\mathbf a'' &      &    & \\ \hline 
-\mathbf a'   & 0             &      &     & \\ \hline
            & 1             & 0   &     & \\ \hline
            &               &  1   &1   & \\ \hline
            &               &      &1   &      \\
            &               &      &0   &      \\
            &               &      &\vdots  & E_{n-s-1} \\      
            &               &      &0    &  \\   
\end{array} \right ) = H_{s,2} \nonumber.
\end{equation}
Using Lemma \ref{lemma2}, we have 
\begin{equation}
Q'_s=H_{s,2}\sim H_{s,n-s}=Q''_s=\left ( 
\begin{array}{c|c|c|c}
D_{s-1}   &-{^t}\mathbf a'' \ 0 \cdots 0 &      &     \\ \hline 
-\mathbf a'          &                                 &      &      \\ 
0          &                                 &       &      \\
\vdots  &   E_{n-s}                     &        &    \\
0          &                                &        &     \\ \hline
            & 0 \cdots 0 \ 1           & 1     &     \\ \hline
            &                                &  1   &0   \\
\end{array} \right ) \nonumber.
\end{equation}
Then by the transformation 
$$T_4= \left (\begin{array}{c|c|c}
I_{n-1} &  &  \\ \hline
         &1 &  \\ \hline
         &-1 &1 \\
\end{array} \right ) ,$$ we have
$$R_s={^t}T_4Q''_sT_4^{(q)}=\left (
\begin{array}{c|c}
D_{s-1} & -{^t}\mathbf a'' \ 0 \cdots 0 \\ \hline
-\mathbf a'        &            \\
0        &             \\
\vdots & E_{n-s+2} \\
0        &                \\
\end{array} \right ) . $$ 
If $s=1$, $R_1\sim W_0$. Suppose that $s>1$. By the transformation
\begin{equation}
T_5=\left ( 
\begin{array}{c|c|c|c|c}
I_{s-1}    &    &     &   & \\ \hline 
            & 1 &     &1  & \\ \hline
 \mathbf a''  &   & 1  &   & \\ \hline
            &    &    &1   &  \\ \hline
            &    &    &   & I_{n-s-1} \\         
\end{array} \right ) \nonumber,
\end{equation}
we obtain
\begin{equation}
R'_s={^t}T_5R_sT_5^{(q)}=\left ( 
\begin{array}{c|c|c|c|c}
D_{s-1}   &      &      &    & \\ \hline 
-\mathbf a'          & 0   &      &     & \\ \hline
            & 1   & 0   &1   & \\ \hline
            &      &  1   &0   & \\ \hline
            &      &      &1   &      \\
            &      &      &0   &      \\
            &      &      &\vdots  & E_{n-s-1} \\      
            &       &      &0    &  \\   
\end{array} \right ) \nonumber.
\end{equation}
If $n-s-1=1$, by the tranformation
$$T_6=\left (
\begin{array}{c|c|c|c}
I_{n-2} & & & \\ \hline
         &1 & & \\ \hline
         &  & 1 & \\ \hline
         &-1& & 1 \\
\end{array}\right ), $$
we have $${^t}T_6R'_{n-2}T_6^{(q)}=B_{n-3}.$$ Suppose that $n-s-1>1$. Then by the transformation 
$$T_7=\left ( 
\begin{array}{c|c|c|c|c|c}
I_s & & & & & \\ \hline
    & 1& & & & \\ \hline
    & & 1& & 1 & \\ \hline
    & -1& & 1& & \\ \hline
    & & & & 1 & \\ \hline
    & & & & & I_{n-s-3} \\
\end{array} \right ), $$
we have 
$$R''_s={^t}T_7R'_sT_7^{(q)}=\left ( 
\begin{array}{c|c|c|c|c|c}
D_{s-1}   &      &                     &     &     &   \\ \hline 
-\mathbf a'  & 0   &                     &     &      & \\ \hline
            &1    &                     &      &     & \\ 
            &0    &       E_2         &      &      & \\ \hline
            &     &   0 \ 1 & 0   & 1  & \\ \hline
            &      &                    &  1  & 0     &\\ \hline
            &     &                      &      & 1  & \\
            &     &                    &         &0  & \\
            &    &                     &          & \vdots & E_{n-s-3} \\
            &    &                    &            &     0  &\\
\end{array} \right ) =H'_{s,2} \nonumber.$$
Using Lemma \ref{lemma3}, we have 
\begin{equation}
R''_s=H'_{s,2}\sim H'_{s,n-s-2}=R'''_s=\left ( 
\begin{array}{c|c|c|c|c|c}
D_{s-1}   &      &                     &     &     &   \\ \hline 
-\mathbf a'  & 0   &                     &     &      & \\ \hline
            &1    &                     &      &     & \\ 
            &0    &                     &      &      & \\
            &\vdots & E_{n-s-2}   &       &     &  \\
            &0    &                  &        &      & \\ \hline
            &     & 0 \cdots 0 \ 1 & 0   & 1  & \\ \hline
            &      &                    &  1  & 0     &\\ \hline
            &     &                      &      & 1  & 0 \\
\end{array} \right ) \nonumber.
\end{equation}
It is easy to see that 
$${^t}T_6R'''_sT_6^{(q)}=B_{s-1}.$$ 
\end{proof}

\begin{lemma}\label{lemma5}
Put 
$$ B_s= \left( 
\begin{array}{c|c}
D_s     &  \\ \hline
\mathbf b_{s}    &    \\
0        &    \\
\vdots &    E_{n-s+1} \\
0        &        \\
\end{array} \right ) , $$
where $s\geq 1,\ n-s+1\geq 1$, $D_s\in M_s(k)$, and $\mathbf b_s$ is a row vector of dimension $s$. Suppose that the rank of $B_s$ is $n$. Then
$$B_s\sim W_s=\left ( 
\begin{array}{c|c}
I_s &  \\ \hline
   & E_{n-s+1}
\end{array} \right ) ,$$
or 
$$B_s\sim B_{s-1}= \left( 
\begin{array}{c|c}
D_{s-1}     &  \\ \hline
\mathbf b_{s-1}      &    \\ 
0            &    \\
\vdots     &    E_{n-s+2} \\
0            &        \\
\end{array} \right ) , $$
where $D_{s-1}\in M_{s-1}(k)$, and $\mathbf b_{s-1}$ is a row vector of dimension $s-1$.
\end{lemma}

\begin{proof}
Suppose that $\det D_s\neq 0$. By Proposition \ref{shimadaprop}, there exists a linear transformation of coordinates $T_D \in GL_s(k)$ such that ${^t}T_D D_sT_D^{(q)}=I_s.$
By the transformation
$$T= \left ( 
\begin{array}{c|c}
T_D &  \\ \hline
    & I_{n-s+1} \\
\end{array} \right ),$$
we have
$${^t}TB_sT^{(q)}=\left (
\begin{array}{c|c}
I_s  &  \\ \hline
\mathbf b'_s &   \\
0    &   \\
\vdots &   E_{n-s+1} \\
0 &  \\
\end{array} \right ), $$
where $\mathbf b'_s=\mathbf b_sT_D^{(q)}$. If $\mathbf b'_s=0$, then $B_s\sim W_s$. Suppose that $\mathbf b'_s\neq 0$. By Lemma \ref{lemma4}, we have $B_s\sim W_s$, or $B_s\sim B_{s-1}$.

Suppose that $\det D_s =0$. Then one row of the matrix $D_s$ is a linear combination of the other rows. By interchanging coordinates $x_0, \cdots ,x_{s-1}$ if nessesary, we can assume that the $s$-th row is a linear combination of the other rows. We write the matrix $D_s$ as 
$$D_s=\left( 
\begin{array}{c|c}
P  & {^t}\mathbf g \\ \hline
\mathbf h  & d \\
\end{array} \right ),$$
where $P\in M_{s-1}(k)$, $\mathbf g,\mathbf h$ are row vectors of dimension $s-1$, $d\in k$ , and that satisfy $\mathbf h=\mathbf wP, d=\mathbf w{^t}\mathbf g$ with $\mathbf w$ being a row vector of dimension $s-1$. Then
$$B_s\sim B'_s =\left (
\begin{array}{c|c|c}
P & {^t}\mathbf g  &  \\ \hline
\mathbf h & d &  \\ \hline
\mathbf f & e &   \\
0 & 0 & \\
\vdots & \vdots & E_{n-s+1} \\
0 & 0 & \\
\end{array} \right ), $$ 
where $\mathbf f$ is a row vector of dimension $s-1$, and $e\in k$. By the tranformation 
$$T' = \left (
\begin{array}{c|c|c}
I_{s-1} & -{^t}\mathbf w & \\ \hline
         & 1       &  \\ \hline
         &          & I_{n-s+1} \\
\end{array} \right ), $$ we obtain
$$B''_s={^t}T'B'_sT'^{(q)}=\left (
\begin{array}{c|c|c}
P & -P\cdot {^t}\mathbf w^{(q)}+{^t}\mathbf g &  \\ \hline
   &                          & \\ \hline
\mathbf f  &-\mathbf f\cdot {^t}\mathbf w^{(q)}+e       &  \\
0 & 0                       & \\
\vdots & \vdots &   E_{n-s+1} \\
0  & 0   &    \\ 
\end{array} \right ). $$
Put $$ Q=\left ( 
\begin{array}{c|c}
P & -P\cdot {^t}\mathbf w^{(q)}+{^t}\mathbf g   \\ \hline
\mathbf f  &-\mathbf f\cdot {^t}\mathbf w^{(q)}+e         \\
\end{array} \right ).$$
Because the rank of $B'_s$ is $n$, we have $\det Q\neq 0$. Let $Q'\in GL_{s}(k)$ such that $QQ'^{(q)}=I_s$,
$$Q'=\left (
\begin{array}{c|c}
P' & {^t}\mathbf g' \\ \hline
\mathbf f' & e' \\
\end{array} \right ), $$
where $P'\in M_{s-1}(k)$, $\mathbf g',\mathbf f'$ are row vectors of dimension $s-1$, $e'\in k$.
By the transformation 
$$T''=\left (
\begin{array}{c|c|c}
P' & {^t}\mathbf g'  & \\ \hline
\mathbf f' & e'    &    \\ \hline
   &      &  I_{n-s+1} \\
\end{array} \right ), $$
we obtain 
$${^t}T''B''_sT''^{(q)}=\left (
\begin{array}{c|c|c}
{^t}P' &  &  \\ \hline
\mathbf g'     & 0  &  \\ \hline 
       & 1  & \\
       & 0  &  \\
       &\vdots & E_{n-s+1} \\
       & 0  &   \\
\end{array}\right ) .$$
Putting $D_{s-1}={^t}P'$ and $\mathbf b_{s-1}=\mathbf g'$, we have $B''_s \sim B_{s-1}$.  
\end{proof}

\begin{remark}\label{remark2}
When $s=1$, we have $$B_{s-1}=B_0=E_{n+1}=W_0.$$
\end{remark}

Now we prove Theorem \ref{maintheorem} and Corollary \ref{cor1}. 
\begin{proof}
Because the rank of the matrix $A$ is $n$, Proposition \ref{shimadaprop} implies that the hypersurface $X_A$ is singular. By using a linear transformation of coordinates if nessesary, we can assume that $X_A$ has a singular point $(0, \cdots , 0,1)$. Then we have $a_{in}=0$ for any $0\leq i \leq n$. The matrix $A$ is now of the form
$$A=\left ( \begin{array}{c|c}
D_n &  \\ \hline
\mathbf b_n &  \\
\end{array} \right )=B_n, $$
where $D_n\in M_n(k)$, and $\mathbf b_n$ is a row vector of dimension $n$. Using Lemma \ref{lemma5} repeatedly and Remark \ref{remark2}, we have that the hypersurface $X_A$ is isomorphic to one of the hypersurfaces defined by the matrixes $W_s$ with $0\leq s\leq n$. 

If $A$ is general, then $\det (D_n)\neq 0$, and hence by the first paragraph of the proof of Lemma \ref{lemma5} and Lemma \ref{lemma4}, we have $A\sim W_{n-1}$.

Next we prove that $s\neq s'$ implies $W_s \not\sim W_{s'}$. For this, we introduce some notions. Let $X_s ^n$ be the hypersurface defined by the matrix $W_s$ in the projective space $\P ^n$. The defining equation of $X_s ^n$ can be written as
$$F_qx_n+ F_{q+1}=0,$$
where 
$$F_q=\begin{cases} 0 \ &\text{if} \ s=n \\ x_{n-1}^q\ & \text{if}\  s<n, \\ \end{cases} $$ and 
$$F_{q+1}=\begin{cases} x_0^{q+1}+\cdots +x_{n-1}^{q+1} \ & \text{if} \ s=n \\ x_0^{q+1}+\cdots +x_{s-1}^{q+1}+x_s^qx_{s+1}+\cdots +x_{n-2}^qx_{n-1}\ & \text{if} \ s<n. \\\end{cases}$$
It is easy to see that $X_s ^n$ has only one singular point $P_0=(0,\cdots , 0,1)$. The variety of lines in $\P ^n$ passing through $P_0$ can be naturally identified with the hypersurface $\mathcal{H}_{\infty }=\{x_n=0\}$ in $\P ^n$ by the correspondence $Q\in \mathcal{H}_{\infty }$ to the line $\overline{QP_0}$. Let $\varphi $ be the map defined by
\begin{eqnarray} \varphi : \P ^{n}\setminus \{P_0\} &\longrightarrow \  \P ^{n-1} \nonumber \\ P &\longmapsto \  \overline{PP_0}. \nonumber \end{eqnarray} 
Let $\overline{X_s ^n}=\varphi (X_s ^n \setminus \{ P_0\} )$.  For $Q=(y_0,\cdots ,y_{n-1},0)\in \mathcal{H}_{\infty }$, we consider the line 
$$l=\overline{QP_0}=\{(\lambda y_0,\cdots ,\lambda y_{n-1},\mu) \mid  (\lambda ,\mu )\in \P ^1\}.$$ 
We have $l\in \overline{X_s^n}$ if and only if there exists $P=(p_0, \cdots , p_{n-1}, p_n) \in X_s ^n\setminus \{ P_0 \}$ satisfying $P\in l$, i.e. there exists an element $\mu \in k$ such that $$(p_0,\cdots , p_{n-1},p_n)=(y_0,\cdots , y_{n-1}, \mu ),$$ for some $P\in X_s^n\setminus \{P_0\}$, or equivalently there exists an element $\mu \in k$ such that  
$$F_q(y_0,\cdots ,y_{n-1})\mu +F_{q+1}(y_0,\cdots ,y_{n-1})=0.$$
Then 
$$\varphi ^{-1}(l)\cap (X_s^n\setminus \{P_0\})=\begin{cases}
\emptyset  & \parbox[t]{4.5cm}{$\text{if} \ F_q(y_0,\dots ,y_{n-1})=0 \ \text{and} \\ \ \ \  F_{q+1}(y_0,\dots ,y_{n-1})\neq 0, $}\\
\{ \text{a single point} \} & \text{if}  \ F_q(y_0,\dots ,y_{n-1})\neq 0, \\
l\setminus \{P_0 \} & \parbox[t]{4.5cm}{$\text{if}\ F_q(y_0,\dots ,y_{n-1})=0 \ \text{and} \\ \ \ \   F_{q+1}(y_0,\dots ,y_{n-1})=0.$}\\
\end{cases}$$
Putting $V_s=\{ F_q=0,\ F_{q+1}=0 \}\subset \P ^{n-1}$, and $H_s=\{ F_q=0 \} \subset \P ^{n-1}$, we have
\begin{equation} \label{maineq} \nonumber
V_s=\begin{cases}
X_s ^{n-2} \ & \text{if} \ s\leq n-2 ,\\ \text{nonsingular Fermat hypersurface in}\ \P ^{n-1}\ & \text{if}\ s=n,\\ 
\text{nonsingular Fermat hypersurface in} \  \P ^{n-2}\ &\text{if}\  s=n-1 ,\end{cases}
\end{equation}
where $X_s ^{n-2}$ is the hypersurface in $\P ^{n-2}$ associated with the matrix 
$$\left ( \begin{array}{c|c} I_s &  \\ \hline  & E_{n-s-1} \\ \end{array} \right ). $$
For any $s\neq s'$, suppose that $X_s^n$ and $X_{s'}^n $ are isomorphic and let $\psi : X_s ^n \longrightarrow X_{s'}^n$ be an isomorphism. 
Because each of $X_s ^n$ and $X_{s'} ^n $ has only one singular point $P_0$, we have $\psi (P_0)=P_0$, and hence $\psi $ induces an isomorphism $\overline{\psi }$ from $\overline{{X_s ^n}}$ to $\overline{X_{s'}^n}$. For any line $l\in \overline{{X_s ^n}}$ and $l'\in \overline{X_{s'}^n}$ such that $\overline{\psi}(l)=l'$, we have 
$$\sharp (\varphi ^{-1}(l)\cap (X_s^n\setminus \{P_0\}))=\sharp (\varphi ^{-1}(l')\cap (X_{s'}^n\setminus \{P_0\})).$$
Thus $V_s\cong V_{s'}$ and $H_s\cong H_{s'}$. Hence for any $s\neq s'$, if $V_s \not\cong V_{s'}$ or $H_s \not\cong H_{s'}$ then $X_s^n \not\cong X_{s'}^n$. 

In the case $n=1$, we have that $X_0^1$ consists of two points, and $X_1^1$ consists of a single point. In the case $n=2$, we have that $X_0^2$ consists of two irreducible components, $X_1^2$ is  irreducible, and $X_2^2$ consists of $(q+1)$ lines. Hence, in the case $n=1$ and $n=2$, we see that $s\neq s'$ implies $W_s \not\sim W_{s'}$. By induction on $n$, we have the proof.  
\end{proof}

Next we prove Corollary \ref{cor2}.

\begin{proof}
Under the condition $n\geq 2,s<n$ and $(n,s)\neq (2,0)$, we have $x_{n-1}$ does not divide $F_{q+1}$, and hence $V_s$ is of codimension 2 in $\P ^{n-1}$. By induction on $n$, $X_s^n$ is irreducible. The morphism 
$$\varphi |_{X_s^n\setminus \{ P_0 \} } : X_s^n\setminus \{ P_0 \} \longrightarrow \mathcal{H}_{\infty } \cong \P ^{n-1} $$
is birational with the inverse rational map 
$$Q=(y_0,\cdots , y_{n-1},0) \longmapsto \left (y_0,\cdots ,y_{n-1}, -\frac {F_{q+1}(y_0,\cdots ,y_{n-1})}{y_{n-1}^q}\right ).$$ 
\end{proof}

\section{Proof of Theorem \ref{auttheorem}}
For any $s\leq n-2$, the matrix $W_s$ can be written
$$W_s= \left (
\begin{array}{c|c|c}
W'_s & & \\ \hline 0 \cdots 0 \ 1 & 0 & \\ \hline & 1 & 0 \\
\end{array} \right ). $$
For any $g\in {\rm Aut}(X_s)$, we have $g(P_0)=P_0$ because $X_s$ has only one singular point $P_0=(0, \cdots ,0,1)$. The automorphism $g$ is defined by a matrix of the form
$$ M=\left ( 
\begin{array}{c|c|c}
T & {^t}\mathbf a & 0 \\ \hline \mathbf b & d & 0 \\ \hline \mathbf c & e & 1 \\
\end{array} \right ) ,$$
where $T\in M_{n-1}(k)$, $\mathbf a, \mathbf b,\mathbf c$ are row vectors of dimension $n-1$, $d,e\in k$.
We have ${^t}MW_sM^{(q)}=\delta W_s$ for some $0\neq \delta \in k$ implies
\begin{numcases}{}
{^t}TW'_sT^{(q)}=\delta W'_s \\
\left [ \mathbf a W'_s+d(0,\cdots ,0,1) \right ] \cdot T^{(q)}=\delta (0,\cdots , 0,1) \\
{^t}TW'_s\cdot {^t}\mathbf a ^{(q)}+{^t}\mathbf cd^q=0 \\
\left [ \mathbf aW'_s+d(0,\cdots ,0,1)\right ] \cdot {^t}\mathbf a ^{(a)}+ed^q=0 \\
\mathbf b=0 \\
d^q=\delta 
\end{numcases}
By (1), we see that $T$ is a matrix defining an automorphism of $X_s ^{n-2}$ in $\P ^{n-2}$. Because $s\leq n-2$, by (2) we have $d=\delta $. Hence we can calculate $T$ by induction on $n$. The vector $\mathbf a, \mathbf c$ and $d,e$ can be find by using the  equations (2)-(6). Conversely, it is easy to show that if the matrix $M$ satifies the conditions (i)-(v) then it define a projective automorphism of $X_s$. The projective automorphism group of $X_n$ and $X_{n-1}$ is easy to calculate. \qed

\section{Proof of Propostion \ref{shimadaprop}}
For the reader's convenience, we give a proof of Proposition \ref{shimadaprop}, which is based on the argument of \cite{Serre88}, chapter VI. The implications (iv)$\Rightarrow $(iii)$\Rightarrow $(ii)$\Rightarrow $(i) are clear. We will prove (i)$\Rightarrow $(iv). For $B \in GL_{n+1}(k)$, consider the map $f_B$ defined by
\begin{eqnarray}\nonumber
 f_B : GL_{n+1}(k)  & \longrightarrow  & GL_{n+1}(k) \\ \nonumber
T & \longmapsto & {^t}TBT^{(q)}.
\end{eqnarray}
Because the differential of the Frobenius map $F: T\longmapsto T^{(q)}$ is identically zero, we can deduce that 
$$d(f_B)=d({^t}T)BT^{(q)}.$$
Therefore, the tangent map of $f_B$ is surjective for any $B\in GL_{n+1}(k)$. Hence, $f_B$ is generically surjective, and the image of $f_B$ contains a non-empty open subset $U_B$. Let $A$ be any matrix of $M_{n+1}(k)$ such that the hypersurface $X_A$ is nonsingular, i.e. $A\in GL_{n+1}(k)$. Because $GL_{n+1}(k)$ is irreducible, we have $U_A \cap U_I \neq \emptyset$, where $I$ is identity matrix of size $n+1$. There exist $T_1,T_2 \in GL_m(k)$ such that $f_A(T_1)=f_I(T_2)$. Putting $T= T_1T_2^{-1}$, we have ${^t}TAT^{(q)}=I$.  \qed
 
\section{The case of plane curves}
Next we will study the plane curves $X_A$ associated with matrices $A$ of rank $\leq 2$ in the projective plane $\P ^2$. 
\begin{theorem}\label{theoinplane}
Let $A=(a_{ij})_{0\leq i,j\leq 2}\in M_3(k)$ be a nonzero matrix and let $X_A$ be the curve defined by $\sum a_{ij}x_ix_j^q=0$ in $\P ^2$. Suppose that the rank of $A$ is smaller than 3. 
\begin{enumerate}
\item[(i)] When the rank of $A$ is 1, the curve $X_A$ is projectively isomorphic to one of the following curves $$Z_0 : x_0^{q+1}=0,\ or \ \ Z_1: x_0^qx_1=0.$$
\item[(ii)] When the rank of $A$ is 2, the curve $X_A$ is projectively isomorphic to one of the following curves
 $$X_0 : x_0^qx_1+x_1^qx_2=0, \ or \ \ 
 X_1: x_0^{q+1}+x_1^qx_2=0,  \ or \ \  
 X_2: x_0^{q+1}+x_1^{q+1}=0.$$
\end{enumerate} 
\end{theorem}

\begin{proof} 
In the case the rank of $A$ is 2. By Theorem \ref{maintheorem}, the plane curve $X_A$ is projectively isomorphic to one of the plane curves $X_0$, or $X_1$, or $X_2$.

In the case rank of $A$ is 1. With the same argument of the proof of Theorem \ref{maintheorem}, we can assume that the matrix $A$ is as following form
$$A=\left ( \begin{array}{ccc}  a_{00} & a_{01}& 0 \\ a_{10}& a_{11}& 0\\ a_{20}& a_{21}& 0 \\ \end{array}\right ).$$
By interchanging with $x_0$ and $x_1$ if nessesary, we can assume that $(a_{01},a_{11},a_{21})\neq (0,0,0)$. Because rank of $A$ is 1, there exists $\lambda \in k$ such that $(a_{00},a_{10},a_{20})=\lambda (a_{01},a_{11},a_{21})$. The curve $X_A$ is defined by the equation
$$(a_{00}x_0+a_{10}x_1+a_{20}x_2)(x_0^q+\lambda x_1^q)=0.$$
It is easy to show that $X_A$ is projectively isomorphic to the curve $Z_0$ or $Z_1$.  
\end{proof}
\begin{remark}\label{remarkHomma}
In fact, the case when the plane curve $X_A$ of degree $p+1$ has been proved by Homma in \cite{Homma}.
\end{remark}
Note that the plane curve $X_1$ has a special property such that the tangent line of $X_1$ at every smooth point passes through the point $(0,1,0)$. Therefore the plane curve $X_1$ is strange. Moreover this curve is  irreducible and nonreflexive. In \cite{BaliHefez}, Ballico and Hefez proved that a reduced irreducible nonreflexive plane curve of degree $q+1$ is isomorphic to one of the following curves:
\begin{enumerate}
\item[(1)] $X_I \ : \ x_0^{q+1}+x_1^{q+1}+x_2^{q+1}=0,$
\item[(2)] a nodal curve whose defining equation is given in \cite{Fuka13} and \cite{HoShi},
\item[(3)] strange curves.
\end{enumerate}  
Let $\mathcal{L}$ be the space of all reduced irreducible projective plane curves of degree $q+1$, which is open in the space $\mathcal{P} \cong \P ^{\binom {q+3}2}$ of all projective plane curves of degree $q+1$. Let $\mathcal{L}_{*}$ be the locus of $\mathcal{P}$ consisting of curves isomorphic to $X_I$, and let $\mathcal{L}_1$ be the locus of $\mathcal{P}$ consisting of strange curves. Let $(\xi _J)$ be the homogeneous coordinates of $\mathcal{P}$ where $J=(j_0,j_1,j_2)$ ranges over the set of all ordered triples on non-negative integer such that $j_0+j_1+j_2=q+1$. The point $(\xi _J)$ corresponds to the curve $\sum \xi _Jx^J=0$ where $x^J=x_0^{j_0}x_1^{j_1}x_2^{j_2}$. Then the locus of all curves defined by the equation of the form $\sum a_{ij}x_ix_j^q=0$ is the linear subspace of $\mathcal{P}$ defined by $\xi _J=0$, unless $J\in \{(q+1,0,0), (0,q+1,0),(0,0,q+1),(q,1,0),(q,0,1),(1,q,0),(1,0,q),(0,q,1),(0,1,q)\}$. By Theorem \ref{theoinplane}, we have that because $Z_0,Z_1,X_0,X_2$ are reducible, the closure $\overline{\mathcal{L}_{*}}$ of $\mathcal{L}_{*}$ in $\mathcal{L}$ consists of curves isomorphic to $X_I$ or to $X_1$, and the intersection of  $\overline{\mathcal{L}_{*}}$ and $\mathcal{L}_1$ consist of curves isomorphic to $X_1$.

\bibliography{myrefs}
\bibliographystyle{jplain}

\end{document}